\documentclass[11pt, twoside]{article}
\usepackage{fancyhdr}
\usepackage{CJK}
\usepackage[a4paper]{geometry}
\usepackage[dvips, usenames]{color}
\usepackage{titletoc}
\usepackage{latexsym}
\usepackage{amsmath}
\usepackage{amssymb}
\usepackage{multicol}
\usepackage{graphics}
\usepackage{graphicx}
\usepackage{subfigure}
\usepackage{indentfirst}
\usepackage{epsfig}
\usepackage{mathrsfs}

\usepackage{amsfonts,amsthm,amssymb}
\usepackage{amsfonts}
\usepackage{color}
\usepackage{ifpdf}
\usepackage{fancyhdr}
\usepackage{epstopdf}
\usepackage{CJK}
\usepackage{float}
\usepackage{footmisc}

\newtheorem{lem}{Lemma}[section]
\newtheorem{thm}[lem]{Theorem}
\newtheorem{cor}[lem]{Corollary}

\newtheorem{con}[lem]{Conjecture}
\newtheorem{Proposition}[lem]{Proposition}

\begin{document}

\setlength\abovedisplayskip{2pt}
\setlength\abovedisplayshortskip{0pt}
\setlength\belowdisplayskip{2pt}
\setlength\belowdisplayshortskip{0pt}

\title{\huge\  A characterization of tight ($ k, 0 $)-stable graphs
	\author{\CJKfamily{fs}\small  Yuqi Xu$^{a}$ \quad Weihua Yang$^{a}$ \quad Xiaxia Guan$^{a*}$\\
		\CJKfamily{kai}\small \textit{$^{a}$Department of Mathematics,\ Taiyuan University of Technology,}\\
		\CJKfamily{kai}\small \textit{Taiyuan, Shanxi\ 030024, China}
	}\date{}}
\maketitle \footnote{
	\textit{$ ^* $E-mail addresses}: guanxiaxia@tyut.edu.cn
	
	This work is supported by National Natural Science Foundation of China (No. 12401462 and 12371356) and the Natural Science Foundation of Shanxi Province (No. 202403021222034). }

\begin{center}\begin{minipage}{14.5cm}

{\CJKfamily{}\small \noindent {\bf Abstract}\ \ Let $ k $ and $ \ell $ be two non-negative integers with $ k>\ell $. A graph $G$ is $(k,\ell)$-stable if $\alpha(G-S) \geq \alpha(G)-\ell$ for every $S \subseteq V(G)$ with $|S|=k$, where $\alpha(G)$ denotes the independence number of $G$. Dong and Wu established that $\alpha(G) \leq \lfloor (n-k+1)/2 \rfloor + \ell$ for a $(k,\ell)$-stable graph $ G $, where $ n $ is the order of $ G $. A $(k,\ell)$-stable graph $ G $ is tight if $\alpha(G) = \lfloor (n-k+1)/2 \rfloor + \ell$. In this paper, we provide a complete characterization of tight $(k,0)$-stable graphs for $ k\geq 4 $. 
In particular, we prove that tight $(k,0)$-stable graphs are $ K_{k+1} $ and	$ K_{k+2} $ for $ k\geq 5 $, which not only extends the result of Liu, Song and Wang {\textit{[J. Graph Theory \textbf{110}(2) (2025), 193--199]}} from $ k\geq 24 $ to $ k\geq 5 $, but also proves the conjecture of Dong and Luo {\textit{[Electron. J. Comb. 32(4) (2025), 4--45]}} once more.
\vskip 2mm \noindent {$Keywords$:}\ $ (k, 0) $-stable graph; independence number }
\end{minipage}\end{center}

\vskip 0.4cm

\section{Introduction}\label{section1}

All graphs considered in this paper are finite, simple, and undirected. Given a graph $G=(V,E)$, we use $V(G)$ and $E(G)$ to denote its vertex set and edge set, respectively. A vertex subset $ S\subseteq V(G) $ is called an \textit{independent set} if $ \{u,v\}\notin E(G) $ for all distinct $ u,v\in S$. Let $\alpha(G)$ denote the \textit{independence number} of $G$, which is the size of a maximum independent set in $G$. 

The resilience of graph properties is a fundamental question in graph theory. Motivated by studies on the Erd\H{o}s--Rogers function, Dong and Wu \cite{DongWu2022} investigated the resilience of graph independence number with respect to removing vertices. For a subset $S\subseteq V(G)$, let $G-S$ denote the induced subgraph of $G$ on vertex set $V(G)- S$. For integers $k > \ell \geq 0$, a graph $G$ is said to be \emph{$(k,\ell)$-stable} if for every subset $S\subseteq V(G)$ with $|S|=k$, we have
\[
	\alpha(G-S) \geq \alpha(G) - \ell.
\]

Dong and Wu~\cite{DongWu2022} established the following upper bound on the independence number of $(k,\ell)$-stable graphs.

\begin{thm}[Dong and Wu \cite{DongWu2022}]\label{thm:DW}
	For all integers $k > \ell \geq 0$, every $(k,\ell)$-stable graph on $n$ vertices satisfies
	$$\alpha(G) \leq \left\lfloor \frac{n-k+1}{2} \right\rfloor + \ell.$$
\end{thm}

A $(k,\ell)$-stable graph $G$ is called \emph{tight} if equality holds in Theorem~\ref{thm:DW}, that is,
$$\alpha(G) = \left\lfloor \frac{n-k+1}{2} \right\rfloor + \ell.$$

Dong and Wu \cite{DongWu2022} proved that there exist tight $n$-vertex $(k,\ell)$-stable graphs for every $n$ and $ \ell $ with $k=\ell+1$ or $k=\ell+2$. Recently, Dong and Luo \cite{DongLuo2024} initiated the study of the structure of tight $(k,0)$-stable graphs. They completely characterized tight $(k,0)$-stable graphs for $k=1$ and $k=2$, showing that such graphs can be arbitrarily large. More precisely, they proved the following structural theorem.

\begin{thm}[Dong and Luo \cite{DongLuo2024}]\label{thm:DL12}
	Let $G$ be a tight $(k,0)$-stable graph on $n$ vertices.
	\begin{itemize}
		\item[(a)] If $k=1$ and $n$ is even, then $G$ contains a perfect matching.
		\item[(b)] If $k=1$ and $n$ is odd, then $G$ has a spanning subgraph that is a vertex-disjoint union of an odd cycle and a (possibly empty) matching.
		\item[(c)] If $k=2$ and $n$ is odd, then $G$ is an odd cycle.
		\item[(d)] If $k=2$ and $n$ is even, then $G$ has a spanning subgraph that is either a vertex-disjoint union of two odd cycles, or an even subdivision of $K_4$.
	\end{itemize}
\end{thm}

In stark contrast to the cases $k=1$ and $k=2$, Dong and Luo \cite{DongLuo2024} also discovered that tight $(k,0)$-stable graphs become very rare when $k \geq 3$. 

\begin{thm}[Dong and Luo \cite{DongLuo2024}]\label{thm:k3}
	Let $G$ be a tight $(3,0)$-stable graph. Then $G$ has a spanning subgraph isomorphic to one of $K_4$, $K_5$, $H_7$, $H_9$, or $T_9$, where $H_7$, $H_9$, and $T_9$ are shown in Figure~\ref{t6}.
\end{thm}

\begin{figure}[htbp]
	\centering
	\includegraphics[width=13cm]{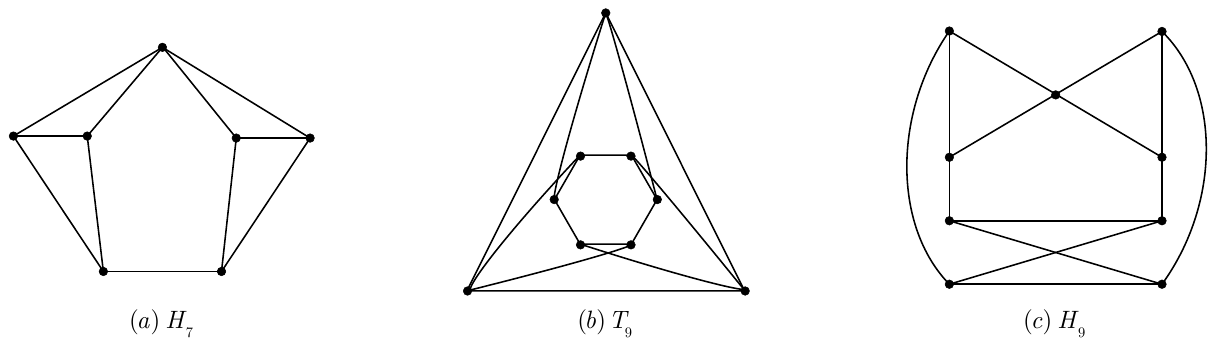}
	\caption{The graphs $H_7$, $H_9$, and $T_9$.}
	\label{t6}
\end{figure}

We collect some basic properties of tight $(k,0)$-stable graphs that follow directly from the definition.

\begin{Proposition}\label{prop1}
	If $G$ is tight $(k,0)$-stable, then for any vertex $v \in V(G)$, the graph $G-v$ is tight $(k-1,0)$-stable.
\end{Proposition}

Hence, by Proposition \ref{prop1} and Theorem \ref{thm:k3}, the order of tight $(k,0)$-stable graphs is obtained.

\begin{cor}[Dong and Luo \cite{DongLuo2024}]\label{thm:DLbound}
	For all $k \geq 3$, every tight $(k,0)$-stable graph has $k+c$ vertices, where $ c\in\{1, 2, 4, 6\} $.
\end{cor}

The complete graph $K_{k+1}$ is always tight $(t,0)$-stable for any $1\le t\le k $. Dong and Luo asked whether there exists any other natural infinite family of tight $(k,0)$-stable graphs for $k \geq 3$. They specifically posed the following question.
 
\begin{con}[Dong and Luo \cite{DongLuo2024}]
	There exists a positive integer $k_0$ such that for all $k \geq k_0$, the only tight $(k,0)$-stable graphs are $K_{k+1}$ and $K_{k+2}$.
\end{con}

This question was subsequently answered affirmatively by Liu, Song, and Wang \cite{LiuSongWang2025}.

\begin{thm}[Liu, Song, and Wang \cite{LiuSongWang2025}]\label{thm:LSW}
	For all $k \geq 24$, the only tight $(k,0)$-stable graphs are $K_{k+1}$ and $K_{k+2}$.
\end{thm}

Despite these advances, the precise structure of tight $(k,0)$-stable graphs for  $4\leq k \leq 23$ remains to be fully determined. In this paper, we provide a complete characterization of tight $(k,0)$-stable graphs. Our main result is as follows.

\begin{thm}\label{thm:main} Let $k\geq 4$, and let $G$ be a graph on $n$ vertices. Then $G$ is a tight $(k,0)$-stable graph if and only if one of the following holds: \begin{enumerate} 
		\item[(i)] $G\cong K_{k+1}$ or $G\cong K_{k+2}$ when $k\geq 5$;
		 \item[(ii)] $G$ is isomorphic to one of $K_5, K_6, H_8^1, H_8^2, H_8^3$ when $k=4$, where $H_8^1, H_8^2, H_8^3$ are shown in Figure \ref{t0}.
\end{enumerate} 
\end{thm} 

\begin{figure}[!htp]
	\centering
	\includegraphics[width=10cm]{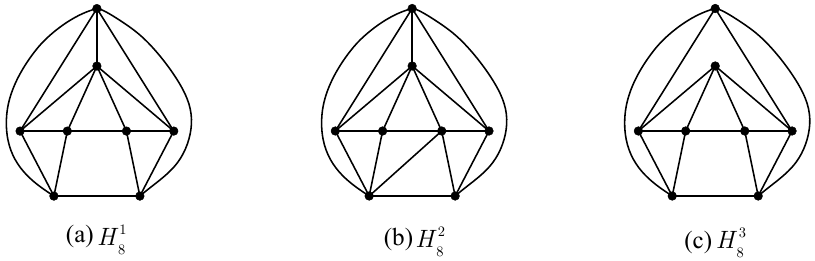}
	\caption{Graphs $ H_8^1 $, $ H_8^2 $ and $ H_8^3 $.}
	\label{t0}
\end{figure}

\section{ Proof of Theorem \ref{thm:main}}

\vskip 2mm
\noindent

In this section, we first establish several important results to prove Theorem \ref{thm:main}.

\begin{Proposition}\label{prop2}
	Let $G$ be a tight $(k,0)$-stable graph with $n = k+c$ vertices. Then for any induced subgraph $H$ of $G$ on $c$ vertices, we have
	$$\alpha(H) = \alpha(G) = \left\lfloor \frac{c+1}{2} \right\rfloor.$$
\end{Proposition}

\begin{proof}
	Let $H=G-S$, where $|S|=k$. Since $G$ is $(k,0)$-stable,
	\[
	\alpha(H)=\alpha(G-S)\geq \alpha(G).
	\]
	Since $H$ is an induced subgraph of $G$, we also have
	\[
	\alpha(H)\leq \alpha(G).
	\]
	Thus equality holds. Finally, as $n=k+c$ and $G$ is tight,
	\[
	\alpha(G)=\left\lfloor \frac{n-k+1}{2}\right\rfloor
	=\left\lfloor \frac{c+1}{2}\right\rfloor .
	\]
\end{proof}

The \emph{Ramsey number} $R(G,H)$ is the minimum integer $N$ such that every graph on $N$ vertices contains either $G$ or $\overline{H}$. Where $\overline{H}$ is the complement of the graph $ H $. We next state two known results for Ramsey number.

\begin{thm}[Greenwood and Gleason \cite{GreenwoodGleason1955}]\label{thm:RG}
	$R(K_4, K_3) = 9$.
\end{thm}

\begin{thm}[Lortz and Mengersen~\cite{LortzMengersen2011}]\label{thm:Ramsey2K3}
	$R(2K_3, K_4)=11$, where $2K_3$ denotes the disjoint union of two triangles.
\end{thm}

\begin{lem}\label{lem:c6}
	Let $G$ be a tight $(k,0)$-stable graph on $n=k+6$ vertices, where $ k\ge4 $. Then $	\alpha(G)=3,$ and $G$ contains neither a copy of $K_4$ nor $2K_3$.
\end{lem}

\begin{proof}
	Since $n=k+6$, by Proposition~\ref{prop2}, for any induced subgraph $ H $ of $ G $ on $ 6 $ vertices, we have
	\[
	\alpha(G)=\left\lfloor \frac{6+1}{2}\right\rfloor=3.
	\]
	
	First, suppose to the contrary that $G$ contains two vertex-disjoint triangles. Then the
	subgraph induced by their six vertices has independence number at most
	$2$, contradicting Proposition~\ref{prop2}. Hence $G$ is
	$2K_3$-free.
	
	Next, suppose that $G$ contains a copy $Q$ of $K_4$. Let
	$x,y\in V(G)\setminus V(Q)$ be arbitrary. Then $G[V(Q)\cup\{x,y\}]$ has
	six vertices and hence has independence number $3$. Since $Q$ is a
	complete graph, an independent set of size $3$ in
	$G[V(Q)\cup\{x,y\}]$ must contain both $x$ and $y$. Therefore
	$xy\notin E(G)$. Since $x,y$ are arbitrary, the set
	$V(G)\setminus V(Q)$ is independent. But this set has size
	\[
	n-4=k+2\geq 6,
	\]
	contradicting $\alpha(G)=3$. Hence $G$ is $K_4$-free.
\end{proof}

Let $ N_G(v) $ be the set of vertices adjacent to $ v $ in $ G $, which are called the  \textit{neighbors} of $ v $. The \textit{degree} of $ v $ in $ G $, denoted by $ d_G(v) $, is the number of the vertices in $ N_G(v) $. We will drop the subscript $ G $ if there is no confusion.

\begin{lem}\label{lem:no-ten}
	There is no tight $(4,0)$-stable graph on $10$ vertices.
\end{lem}

\begin{proof}
	Suppose to the contrary that such a graph $G$ exists. Then
	\[
	\alpha(G)=\left\lfloor \frac{10-4+1}{2}\right\rfloor=3,
	\]
	Moreover, by Proposition~\ref{prop2}, for every $X\subseteq V(G)$ with $|X|=6$, we have
	\[
	\alpha(G[X])=3.
	\]
	By Lemma \ref{lem:c6}, $ G $ is $ K_4 $-free and $ 2K_3 $-free.
	
	Choose a vertex $v_1\in V(G)$ with
	\[
	d_G(v_1)=\Delta(G).
	\]
	
    By Proposition~\ref{prop1}, the graph $G-v_1$ is a tight $(3,0)$-stable graph on $9$
	vertices. Hence $G-v_1$ must contain a spanning subgraph isomorphic to either $T_9$ or $H_9$ (as shown in Figure \ref{t6}(a) and (b)).
	
	We first suppose that $ G-v_1 $ contains a spanning subgraph isomorphic to $ T_9 $, labelled as in Figure~\ref{t1}. Note that $T_9$ contains the triangle $ v_2v_9v_{10} $ and $ G $ is $2K_3$-free, then any edge with ends of distance 2 on cycle $ C_{v_3v_4v_6v_8v_7v_5} $ is not in $ E(G) $, otherwise results a triangle disjoint from $ v_2v_9v_{10} $, a contradiction. That is, $ \{v_3, v_6, v_7\} $ and $ \{v_4, v_5, v_8\} $ are independent sets in $ G $. Moreover, the vertex $ v_1 $ cannot be adjacent to both ends of any edge in cycle $ C_{v_3v_4v_6v_8v_7v_5} $ that would form, together with $ v_1 $, a triangle disjoint from $ v_2v_9v_{10} $.  In particular,
	
	\begin{equation}\label{e1}
		 |N_G(v_1)\cap \{v_3,v_4,v_5,v_6,v_7,v_8\}|\leq3,
	\end{equation}	
 and 
 \begin{equation}\label{e2}
 	\begin{split}
 	&|N_G(v_1)\cap \{v_3,v_4,v_5,v_6,v_7,v_8\}|=3,~\text{if and only if}\\
 	&N_G(v_1)\cap \{v_3,v_4,v_5,v_6,v_7,v_8\}=\{v_3, v_6, v_7\}~\text{or}\\
 	&N_G(v_1)\cap \{v_3,v_4,v_5,v_6,v_7,v_8\}=\{v_4, v_5, v_8\}.
 	\end{split}
 \end{equation}

	\begin{figure}[!htp]
		\centering
		\includegraphics[width=4cm]{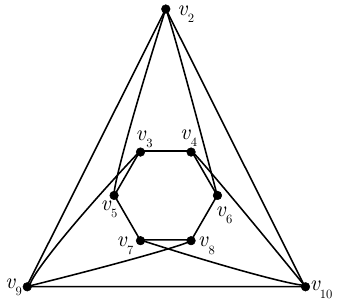}
		\caption{The labelled graph $ T_9 $.}
		\label{t1}
	\end{figure}
	
	Thus $d_G(v_1)\leq 6$. Note that $d_G(v_1)\geq4$, since $ \Delta(T_9)=4 $. In fact, $ d_G(v_1)\geq5 $.
	
	If $d_G(v_1)=4=\Delta(G)$, then $ N_G(v_1)\cap \{v_2,v_9,v_{10}\}=\emptyset $ as $d_{T_9}(v_2)=d_{T_9}(v_9)=d_{T_9}(v_{10})=4 $. Hence $ N_G(v_1)\cap \{v_3,v_4,v_5,v_6,v_7,v_8\}=4 $, which contradicts Inequality~(\ref{e1}).
	
	Moreover, 
	\[ N_G(v_1)\cap \{v_3, v_6, v_7\}\neq\emptyset \quad\text{and}\quad N_G(v_1)\cap \{v_4, v_5, v_8\}\neq\emptyset. \tag{3}
	\]
	Otherwise, if $ N_G(v_1)\cap \{v_3, v_6, v_7\}=\emptyset $, then $ \{v_1, v_3, v_6, v_7\} $ is an independent set in $ G $, which contradicts the condition that $ \alpha(G)=3 $. Similarly, $ N_G(v_1)\cap \{v_4, v_5, v_8\}=\emptyset $.
	
	Inequalities (1)-(3) implies that
	
	(i) $d_G(v_1)\neq 6$,
	
	(ii) $d_G(v_1)=5$, and we may assume that $ N_G(v_1) \cap \{v_3, v_4, v_5, v_6, v_7, v_8\}=\{v_3, v_8\} $ by automorphism. Thus
	\[
	 N_G(v_1)=\{v_2, v_3, v_8, v_9, v_{10}\}, 
	\]
	corresponding to the graph $ T_{10} $ shown in Figure~\ref{g2}. However, $ \alpha(T_{10}[\{v_1, v_2, v_3, v_4, v_6, v_9\}])=2 $, a contradiction. 
	
	\begin{figure}[!htp]
		\centering
		\includegraphics[width=4cm]{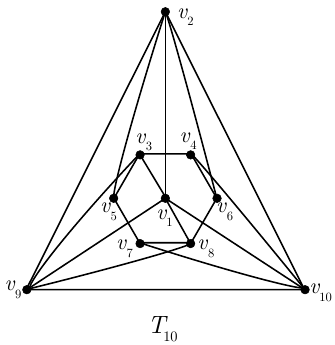}
		\caption{Graph $ T_{10} $. }
		\label{g2}
	\end{figure}

    This contradiction excludes the case in which $G-v_1$ contains a spanning copy of $T_9$.
	
	It remains to consider the possibility that $ G-v_1 $ contains a spanning subgraph isomorphic to $ H_9 $, labelled as in Figure~\ref{g3}. Note that $d_G(v_1)\geq4$, since $ \Delta(H_9)=4 $. We now show that $d_G(v_1)=4$. Observe that $H_9$ contains the triangles $ v_2v_4v_5 $ and $ v_3v_4v_6 $. Since $ G $ is $ 2K_3 $-free, 
	\[ |N_G(v_1)\cap V(C_{v_2v_5v_7v_{10}v_9})|\leq2\quad\text{and}\quad |N_G(v_1)\cap V(C_{v_3v_6v_8v_9v_{10}})|\leq2.
	\]
	Without loss of generality, assume that 
	\[ N_G(v_1)\cap V(C_{v_2v_5v_7v_{10}v_9})=\{v_2, v_7\},~ \{v_5, v_{10}\}~ \mathrm{or}~ \{v_7, v_9\}.
	\]
(i) If $ N_G(v_1)\cap V(C_{v_2v_5v_7v_{10}v_9})=\{v_2, v_7\} $, then 
\[
v_1v_5,~v_1v_8,~v_1v_9,~v_1v_{10}\notin E(G),
\]	
	and 
	\[
	|N_G(V_1)\cap\{v_3, v_6\}|\leq1.
	\]

\noindent(ii) If $ N_G(v_1)\cap V(C_{v_2v_5v_7v_{10}v_9})=\{v_5, v_{10}\} $, then 
\[
v_1v_2,~v_1v_3,~v_1v_7,~v_1v_9\notin E(G),
\]	
and 
\[
|N_G(V_1)\cap\{v_6, v_8\}|\leq1.
\]

\noindent(iii) If $ N_G(v_1)\cap V(C_{v_2v_5v_7v_{10}v_9})=\{v_7, v_9\} $, then 
\[
v_1v_2,~v_1v_5,~v_1v_8,~v_1v_{10}\notin E(G),
\]	
and 
\[
|N_G(V_1)\cap\{v_3, v_6\}|\leq1.
\]
	Therefore
	\[
	|N_G(v_1)\setminus\{v_4\}|\leq 3,
	\]
	and hence
	\[
	d_G(v_1)\leq 4.
	\]
	
	Thus $d_G(v_1)=4$, and $ v_4\in N_G(v_1), $
	that is,
	\[
	v_1v_4\in E(G).
	\]
	This implies $\Delta(G)\geq d_G(v_4)\geq d_{H_9}(v_4)+1=5$, which contradicts the fact that $ \Delta(G)=4 $.
	
	Thus the case in which $ G-v_1 $ contains a spanning copy of $ H_9 $ is impossible. Therefore no tight $(4,0)$-stable graph on $10$ vertices exists.
\end{proof}

\begin{figure}[!htp]
\centering
\includegraphics[width=3cm]{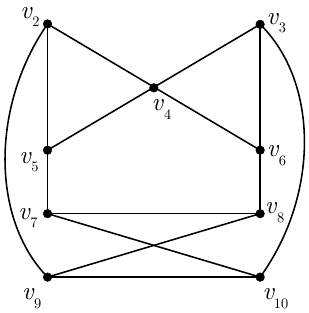}
\caption{Graph $ H_9 $. }
\label{g3}
\end{figure}

\begin{lem}\label{lem:eig}
	Let $G$ be a tight $ (4, 0) $-stable graph with order 8.
	Then $G$ is isomorphic to one of $H_8^1, H_8^2, H_8^3$ shown in Figure \ref{t0}.
\end{lem}

\begin{proof}
	Clearly, $H_8^1, H_8^2, H_8^3$ are tight $ (4, 0) $-stable graphs.
Since $G$ is tight $(4,0)$-stable, $ \alpha(G)=\alpha(G-S)=2 $ for every $ S\subseteq V(G) $ with $ |S|=4 $ by Proposition~\ref{prop2}. 

Choose a vertex $v_1\in V(G)$ with
\[
d_G(v_1)=\Delta(G).
\]
By Proposition~\ref{prop1}, $G-v_1$ is a tight
$(3,0)$-stable graph on $7$ vertices. Hence, by Theorem \ref{thm:k3}, $G-v_1$ contains a spanning subgraph
isomorphic to $H_7$. We label this copy as in Figure~\ref{t4}(a). 

	Clearly, $ \Delta(G)\geq \Delta(H_7)\geq 4$. Both $ v_3v_5v_7 $ and $ v_4v_6v_8 $ are triangles in $ H_7 $. Since $ G $ is $ K_4 $-free, 
	\[ |N_G(v_1)\cap \{v_3, v_5, v_7\}|\leq2\quad\text{and}\quad |N_G(v_1)\cap \{v_4, v_6, v_8\}|\leq2.  \tag{4}
\]
Hence
\[
d_G(v_1)\leq 5.
\]
	Therefore,
\[
\Delta(G)\in\{4,5\}.
\]

We now distinguish two cases.

\textbf{Case 1. $\Delta(G)=5$.}

By above discussion, $ v_2\in N_G(v_1) $. Note that $ v_3v_5v_7 $ and $ v_2v_4v_6 $ are triangles in $ H_7 $. Then $ |N_G(v_1)\cap \{v_3, v_5\}|\leq1 $ and $ |N_G(v_1)\cap \{v_4, v_6\}|\leq1 $ since $ G $ is $ K_4 $-free. Without loss of generality, we assume that $v_1$ is not adjacent to $v_5$ and $v_6$ by automorphism. Thus $ N_G(v_1)=\{v_2,v_3,v_4,v_7,v_8\} $. 

Since \(\alpha(G)=2\), we must have \(v_5v_6\in E(G)\); otherwise
\(\{v_1,v_5,v_6\}\) would be an independent set in $ G $.  Furthermore, the
\(K_4\)-freeness of \(G\) implies that
\[
v_3v_4,\ v_2v_7,\ v_2v_8,\ v_4v_7,\ v_3v_8,\ v_3v_6,\ v_4v_5 \notin E(G).
\]
and $|\{v_5v_8, v_6v_7\} \cap E(G)| \leq 1$. 
If $|\{v_5v_8, v_6v_7\} \cap E(G)| =0 $, then $ G\cong H_8^1 $. Assume that $|\{v_5v_8, v_6v_7\} \cap E(G)| =1 $. Clearly, $ H_8^2=H_8^1+v_6v_7\cong H_8^1+v_5v_8 $. Hence $G\cong H_8^1$ or $G\cong H_8^2$.

\begin{figure}[!htp]
	\centering
	\includegraphics[width=12cm]{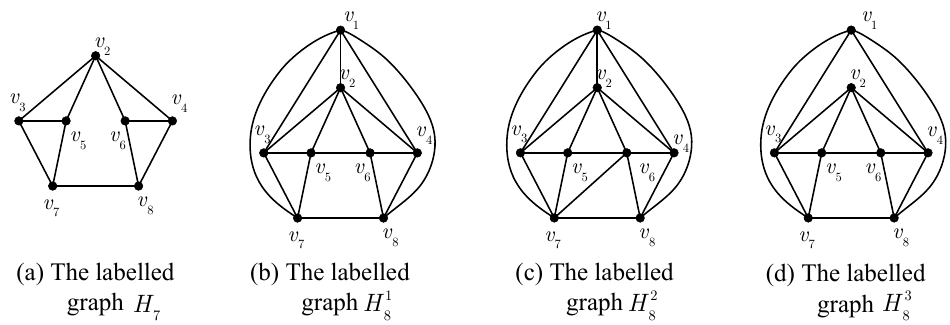}
	\caption{Graphs $ H_7, H_8^1 $, $ H_8^2 $ and $ H_8^3 $.}
	\label{t4}
\end{figure}

\textbf{Case 2.} $\Delta(G) = 4$.

Note that $ d_{H_7}(v_2)=4 $. Then $ v_2\notin N_G(v_1), v_2v_7, v_2v_8\notin E(G) $. Hence $ \{v_7, v_8\}\subseteq N_G(v_1) $. Otherwise, $ \{v_1, v_2, v_7\} $ or $ \{v_1, v_2, v_8\} $ is an independent set in $ G $.  By Inequality $ (4) $ and automorphism, we may assume $ N_G(v_1)=\{v_3, v_4, v_7, v_8\} $, since $d_G(v_1)=\Delta(G) = 4$, see $ H_8^3 $ shown in Figure \ref{t4}(d). It is easy to see $ d_{H_8^3}(v_i)=\Delta(G) = 4 $ for all $ v_i\in\{v_1, v_2, \ldots, v_6\} $. However $ v_7v_8\in E(G) $, thus $G\cong H_8^3$. 

In conclusion, the result holds.
\end{proof}

We now prove Theorem \ref{thm:main}.

\begin{proof}[Proof of Theorem~\ref{thm:main}]
Let $G$ be a tight $(k,0)$-stable graph on $n=k+c$ vertices. By Theorem~\ref{thm:k3},
\[
c\in\{1,2,4,6\}.
\]

If $c=1$ or $c=2$, then
\[
\alpha(G)=\left\lfloor \frac{c+1}{2}\right\rfloor=1.
\]
Thus $G$ is complete. Hence $G\cong K_{k+1}$ when $c=1$, and
$G\cong K_{k+2}$ when $c=2$.

Now assume that $c=4$. Then
\[
\alpha(G)=\left\lfloor \frac{5}{2}\right\rfloor=2.
\]
By Proposition~\ref{prop2}, every induced subgraph of $G$ on
four vertices has independence number exactly $2$. In particular, $G$ is
$K_4$-free.

If $k\geq 5$, then $n=k+4\geq 9$. By Theorem~\ref{thm:RG}, every
graph on at least nine vertices contains either a $K_4$ or an independent
set of size $3$. This contradicts the facts that $G$ is $K_4$-free and
$\alpha(G)=2$. Hence $k=4$. By Lemma~\ref{lem:eig}, the conclusion holds.

Finally, assume that $c=6$. Then Lemma~\ref{lem:c6} implies
that $G$ is both $K_4$-free and $2K_3$-free. By Lemma~\ref{lem:no-ten}, $k\geq 5$. Then
$n=k+6\geq 11$. Since $\alpha(G)=3$, the complement $\overline{G}$ contains no copy of $K_4$.
Applying Theorem~\ref{thm:Ramsey2K3} to $\overline{G}$, we conclude that $G$ contains a copy of $2K_3$.
This contradicts Lemma~\ref{lem:c6}. Hence, $c\neq6$. This completes the proof.
\end{proof}

\clearpage

\end{document}